\documentclass[12pt]{amsart}

\usepackage{psfrag}
\usepackage{color}

\usepackage{graphicx,graphics}
\usepackage{fullpage,amssymb,amsfonts,amsmath,amstext,amsthm,amscd,enumerate,verbatim,tikz}
\usepackage[T1]{fontenc}
\usetikzlibrary{matrix,arrows}

\begin{document}

\newtheorem{theorem}{Theorem}[section]
\newtheorem{result}[theorem]{Result}
\newtheorem{fact}[theorem]{Fact}
\newtheorem{conjecture}[theorem]{Conjecture}
\newtheorem{definition}[theorem]{Definition}
\newtheorem{lemma}[theorem]{Lemma}
\newtheorem{proposition}[theorem]{Proposition}

\newtheorem{corollary}[theorem]{Corollary}
\newtheorem{facts}[theorem]{Facts}
\newtheorem{question}[theorem]{Question}
\newtheorem{props}[theorem]{Properties}
\newtheorem{ex}[theorem]{Example}
\newtheorem*{remark}{Remark}

\newcommand{\notes} {\noindent \textbf{Notes.  }}
\newcommand{\note} {\noindent \textbf{Note.  }}
\newcommand{\defn} {\noindent \textbf{Definition.  }}
\newcommand{\defns} {\noindent \textbf{Definitions.  }}
\newcommand{\x}{{\bf x}}
\newcommand{\z}{{\bf z}}
\newcommand{\B}{{\bf b}}
\newcommand{\V}{{\bf v}}
\newcommand{\T}{\mathcal{T}}
\newcommand{\Z}{\mathbb{Z}}
\newcommand{\Hp}{\mathbb{H}}
\newcommand{\D}{\mathbb{D}}
\newcommand{\R}{\mathbb{R}}
\newcommand{\N}{\mathbb{N}}
\renewcommand{\B}{\mathbb{B}}
\newcommand{\C}{\mathbb{C}}
\newcommand{\dt}{{\mathrm{det }\;}}
 \newcommand{\adj}{{\mathrm{adj}\;}}
 \newcommand{\0}{{\bf O}}
 \newcommand{\av}{\arrowvert}
 \newcommand{\zbar}{\overline{z}}
 \newcommand{\htt}{\widetilde{h}}
\newcommand{\ty}{\mathcal{T}}
\renewcommand\Re{\operatorname{Re}}
\renewcommand\Im{\operatorname{Im}}
\newcommand{\diam}{\operatorname{diam}}

\newcommand{\ds}{\displaystyle}
\numberwithin{equation}{section}

\renewcommand{\theenumi}{(\roman{enumi})}
\renewcommand{\labelenumi}{\theenumi}

\title{The Moduli space of Riemann Surfaces of Large Genus}

%\author{Alastair Fletcher, Jeremy Kahn \& Vladimir Markovic}

\author{Alastair Fletcher}
\address{Department of Mathematical Sciences, Northern Illinois University, Dekalb, IL 60115, USA}
\email{fletcher@math.niu.edu}

\author{Jeremy Kahn}
\address{Department of Mathematics, Brown University, Providence, RI 02912, USA}
\email{kahn@math.brown.edu}

\author{Vladimir Markovic}
\address{Department of Mathematics, Caltech, Pasadena, CA 91125, USA}
\email{markovic@caltech.edu}

\begin{abstract}
Let $\mathcal{M}_{g,\epsilon}$ be the $\epsilon$-thick part of the moduli space $\mathcal{M}_g$ of closed genus $g$ surfaces.
In this article, we show that the number of balls of radius $r$ needed to cover $\mathcal{M}_{g,\epsilon}$ is bounded below by $(c_1g)^{2g}$ and bounded above by $(c_2g)^{2g}$, where the constants $c_1,c_2$ depend only on $\epsilon$ and $r$, and in particular not on $g$. Using this counting result we prove that there are  Riemann surfaces of arbitrarily large injectivity radius that are not close (in the Teichm\"uller metric) to a finite cover of a fixed closed Riemann surface. This result illustrates  the sharpness of the Ehrenpreis conjecture.

\end{abstract}

\maketitle

\section{Introduction}

\subsection{The covering number}
Let $(X,d)$ be a metric space. A natural question to ask is the following: given a set $E\subset X$, how many balls of radius $r$ does it take to cover $E$? It is an easy exercise to check that if $X=\R^n$ with the Euclidean norm and $E$ is a ball of radius $R$, then it takes $C(R/r)^n$ balls of radius $r$ to cover $E$, for some constant $C$ depending only on $n$. 
To make this notion more precise, we introduce the following definition.

\begin{definition}
Let $(X,d_X)$ be a metric space and $E \subset X$. Then the \emph{$r$-covering number} $\eta_{X}(E,r)$ is the minimal number of balls in $X$ of radius $r$ needed to cover $E$ in $X$.
\end{definition}

We will usually suppress the $r$ and just use the terminology \emph{covering number}, if the context is clear.
Observe that $\eta_X$ may be infinite: consider the example of $X=l^{\infty}$, $E$ the unit ball in $X$ and $r<1$.

The following two elementary facts about covering numbers will be used throughout the paper.
If $E \subset F$ in $X$, then
\begin{equation}
\label{obvfact}
\eta_{X} (E,r) \leq \eta_{X} (F,r).
\end{equation}

Let $E \subset Y \subset X$. Then considering $Y$ as the metric space $(Y,d_X)$ we have
\begin{equation}
\label{obvfact-1}
\eta_{X} (E,r) \leq \eta_{Y} (E,r) \leq \eta_X(E,r/2).
\end{equation}

In this paper, we are interested in estimating the covering number for the thick part of moduli space viewed as a subset of moduli space. 
For $g \geq 2$, let $\mathcal{M}_g$ be the moduli space of closed genus $g$ surfaces with its Teichm\"uller metric. 
Fix once and for all $\epsilon >0$, and denote by $\mathcal{M}_{g,\epsilon}$ the $\epsilon$-thick part of moduli space. For simplicity we let 
$$
\eta( \mathcal{M}_{g,\epsilon}, r) =\eta_{\mathcal{M}_{g}} ( \mathcal{M}_{g,\epsilon}, r).
$$
The first main theorem of this paper is as follows.
\begin{theorem}
\label{mainthmmg}
Let $\epsilon >0$ and $r>0$. Then there exists constants $c_1,c_2>0$ which depend only on $\epsilon$ and $r$ such that
\begin{equation} 
\label{mainthmeq}
(c_1g)^{2g} \leq \eta ( \mathcal{M}_{g,\epsilon}, r) \leq (c_2g)^{2g},
\end{equation}
for all large $g$.
\end{theorem}

\subsection{Covering  the Moduli space by balls}

Throughout, we denote by $B_X(p,r)$ the open ball of radius $r$, centered at $p \in X$, in the metric space $(X,d)$. 

We prove Theorem  \ref{mainthmmg} by first showing that each of the two inequalities in \eqref{mainthmeq} holds for a particular value of $r$. The following upper bound on the covering number is in spirit  similar to the one obtained in \cite{KM} on the number of homotopy classes of essential surfaces in a given hyperbolic 3-manifold. 

\begin{theorem}
\label{upperbound}
Let $\epsilon >0$. Then there exist $r_u>0$ and $c_u>0$ such that
\[ \eta (\mathcal{M}_{g,\epsilon},r_u) \leq (c_ug)^{2g}\]
for all large $g$.
\end{theorem}

Next, we have the lower bound:

\begin{theorem}
\label{lowerbound}
Let $\epsilon >0$. Then there exist $r_l>0$ and $c_l>0$ such that
\[ \eta (\mathcal{M}_{g,\epsilon},r_l) \geq (c_lg)^{2g}\]
for all large $g$.
\end{theorem}

Once we have each of the inequalities in \eqref{mainthmeq} for particular values of $r$, we will need to extend the inequalities obtained in Theorems \ref{upperbound} and \ref{lowerbound} to all values of $r$. The key tool here is Theorem \ref{thmteich} below. We first fix some notation. 

Let $S_g$ be a closed topological surface of genus $g\geq 2$. Recall that $\mathcal{T}_g$ is the set of marked Riemann surfaces $S$ where the marking is given by the homotopy class of a homeomorphism $f:S_g \to S$. We often suppress the marking and simply say $S \in \T_g$.  Recall that the Teichm\"{u}ller metric $d_{\T}$ is defined by
\[ d_{\T}(S_1,S_2) = \inf \left \{  \log \sqrt{K} : f:S_1 \to S_2 \text{ is $K$-quasiconformal} \right \},\]
see for example \cite{FMbook}. 

\begin{theorem}
\label{thmteich}
Let $S \in \T_g$ denote a  Riemann surface with injectivity radius $\epsilon >0$. Let $R>0$ and let $B_{\T_{g}}(S,R)$ be the ball in $\T_g$ with respect to the Teichm\"{u}ller metric on $\T_g$. 
Then there are constants $d_1=d_1(R,r) \ge 0$ and $d_2=d_2(\epsilon,R,r)>1$ (in particular, $d_1$ and $d_2$ do not depend on the genus $g$) such that
\[ 
d^{g}_1 \leq \eta_{\T_{g}}(B_{\T_{g}}(S,R),r) \leq d^{g}_{2}, 
\]
where $d_1(R,r) \to \infty$ as $r \to 0$, for any fixed $R>0$.
\end{theorem}

\begin{remark} In the proof of Theorem \ref{mainthmmg} we only use the second inequality from Theorem \ref{thmteich}. The first inequality will be used in the proof of Theorem \ref{thm-druga} below. 
\end{remark}

To prove Theorem \ref{thmteich} we reduce the problem of estimating the counting number in a highly non-trivial metric space like the Teichm\"{u}ller space to the same problem in the Bers space of holomorphic quadratic differentials. This is done using the Bers embedding theorem and employing techniques similar to those in \cite{F2}.

\subsection{Proof of Theorem \ref{mainthmmg}} With the intermediate results Theorems \ref{upperbound}, \ref{lowerbound} and \ref{thmteich} in hand, the proof of Theorem \ref{mainthmmg} runs as follows:

Assume the results of Theorems \ref{upperbound} and \ref{lowerbound} hold. It is clear that if a set can be covered by a certain number of balls of radius $r$, then it can be covered by the same number of balls of larger radius, that is, the covering number satisfies
\[ \eta_X(E,r) \leq \eta_X(E,s),\]
for $r \geq s$. In particular, if $r\geq r_u$, Theorem \ref{upperbound} implies that  
\[ \eta(\mathcal{M}_{g,\epsilon},r) \leq (c_ug)^{2g}.\]
We also notice that if a certain number of balls of radius $r$ are needed to cover a set $E$, then at least that many balls of smaller radius are needed to cover $E$.
In particular, if $r\leq r_l$, Theorem \ref{lowerbound} implies that
\[ \eta (\mathcal{M}_{g,\epsilon},r) \geq (c_lg)^{2g}.\]

It remains to prove the upper bound on $\eta(\mathcal{M}_{g,\epsilon},r)$ when $r<r_u$ and the lower bound on  $\eta(\mathcal{M}_{g,\epsilon},r)$ when  $r>r_l$.
Suppose first that $r<r_u$. If $B_{\T_{g}}(S,r_u)$ is a ball in Teichm\"{u}ller space of radius $r_u$, where $S$ has injectivity
radius at least $\epsilon$, then Theorem \ref{thmteich} implies that
\[ \eta_{\T_{g}}(B_{\T_{g}}(S,r_u),r) \leq d^{g}_2, \]
where $d_2=d_2(\epsilon,r)$. Projecting to moduli space from Teichm\"{u}ller space does not increase distances, and so an upper bound for $\eta(\mathcal{M}_{g,\epsilon},r)$ is provided by multiplying the minimal number of balls or radius $r_u$ needed to cover $\mathcal{M}_{g,\epsilon}$, multiplied by the number of balls of radius $r$ needed to cover a ball of radius $r_u$. That is, we have
\begin{align*} 
\eta(\mathcal{M}_{g,\epsilon},r) &\leq \eta_{\T_{g}}(B_{\T_{g}}(S,r_u),r) \cdot \eta(\mathcal{M}_{g,\epsilon},r_u) \\
& \leq d^{g}_2 (c_ug)^{2g} <(c_2 g)^{2g} 
\end{align*}
for $r<r_u$ and some constant $c_2=c_2(\epsilon,r)$. 

A similar argument shows that if $r>r_l$ then
\begin{align*}
\eta(\mathcal{M}_{g,\epsilon},r) & \geq \frac{\eta(\mathcal{M}_{g,\epsilon},r_l)}{\eta_{\T_{g}}(B_{\T_{g}}(S,r),r_l)} \\
&\geq d^{-g}_1 (c_lg)^{2g}>(c_1 g)^{2g},
\end{align*}
for some $c_1=c_1(\epsilon,r)>0$. This completes the proof.

\subsection{Thick Riemann surfaces and the Ehrenpreis Conjecture} 

The recently proved Ehrenpreis Conjecture (see \cite{k-m-1}) states that given two closed Riemann surfaces $S$ and $M$ and any $K>1$, one can find finite (unbranched) covers $S_1$ and $M_1$, of $S$ and $M$ respectively, such that there exists a $K$-quasiconformal map  $f:S_1 \to M_1$. One of many equivalent formulations of this result states that given any $\xi>0$ and a closed Riemann surface $S$, there exists a finite cover $S_1$ of $S$, such that $S_1$ admits a  tiling into $\xi$-nearly equilateral right angled hexagons (a $\xi$-nearly equilateral right angled hexagon is a polygon that is $(1+\xi)$-quasi-isometric to the standard equilateral right angled hexagon). Or more generally, $S$ has a finite cover that can be tiled by polygons that are small perturbations of some fixed polygon (which we also call a pattern) that represents a fundamental domain of some  closed Riemann surface or an 
 orbifold.

Given $S$, a pattern and $\xi>0$, the resulting cover $S_1$ will typically have large injectivity radius (this can be enforced by the choice of a pattern). Naturally one can ask if the secret as to why the Ehrenpreis conjecture holds is because  once we fix $\xi>0$ and a pattern then any  closed Riemann surface $X$ of sufficiently large injectivity radius can be tiled into polygons that are $\xi$-close to the given pattern. We formulate the following questions:

\begin{question}\label{q1} Let $S_0$ denote a closed Riemann surface or an orbifold whose Euler characteristic $\chi(S_0)$ satisfies  $|\chi(S_0)| \le 2$. 
Is there a function $I(\delta)>0$, $\delta>0$, such that every closed Riemann surface $X$, whose injectivity radius at every point is greater than $I(\delta)$, is at a distance $\leq \delta$ from a finite cover of $S_0$. 
\end{question}

\begin{question}\label{q2} Let $S_0$ denote a closed Riemann surface or an orbifold whose Euler characteristic $\chi(S_0)$ satisfies  $|\chi(S_0)| \le 2$. Are there constants  $I=I(S_0)>0$ and $d=d(S_0)>0$, such that every closed Riemann surface $X$, whose injectivity radius at every point is greater than $I$, is at a distance $\leq d$ from a finite cover of $S_0$. 
\end{question}

\begin{remark} The assumption on the Euler characteristic is made to ensure that $S_0$ has covers of every genus.
\end{remark}

It is well known in the spectral theory of Riemann surfaces that there are Riemann surfaces of arbitrarily large injectivity radius (about every point), with a uniform lower bound on the first eigenvalue of the Laplacian. More precisely, the following statement holds:

\begin{itemize}
\item There exists a universal constant $q_0>0$ such that given any $I>1$,  there exists a Riemann surface  $X$ whose injectivity radius at every point is greater than $I$, and the first eigenvalue $\lambda(X)$ of the Laplacian is greater than $q_0$.
\end{itemize}

Such surfaces $X$ were constructed in \cite{Brooks}. In fact, we can take $q_0$ to be any number between $0$ and $\frac{1}{4}$.

Now, let $S_0$ be a genus $2$ surface with first eigenvalue $\lambda(S)>0$. If $S$ is a finite cover of $S_0$ then $\lambda(S) \le \lambda(S_0)$. Assuming that the Teichm\"uller distance between $X$ and $S$ is at most $\log K$ implies (see Theorem 14.09.02 in \cite{Buser}) that
$$
\lambda(X)\le K^4 \lambda(S)\le K^4  \lambda(S_0).
$$
Thus if we choose $S_0$ such that $\lambda(S_0)$ is sufficiently small (which is ensured by choosing $S_0$ to have a sufficiently short separating curve), then no cover of $S_0$ can be at a distance at most $\log K$ from $X$.

We see that both questions we stated have negative answers. But using our counting techniques we can prove much more.  
By $\T_{g,\epsilon}$ we denote the $\epsilon$-thick part of the Teichm\"uller space $\T_g$ (that is $\T_{g,\epsilon}$ is the cover of $\mathcal{M}_{g,\epsilon}$)

\begin{theorem}\label{thm-druga} Let $S_0$ be a closed Riemann surface or an orbifold. There exists a universal constant $\delta_0>0$ such that every ball of radius $1$ in every $\T_g$ contains a Riemann surface $X$ that is at least distance $\delta_0$ away from any finite cover of $S_0$.
\end{theorem}

\begin{proof} Let $B_{\T_{g}}(S,1)$ denote the ball of radius $1$ that lives inside of $\T_g$. Then by the lower bound from Theorem \ref{thmteich} we know that given any $D>0$  there exists a small enough $r=r(D)>0$ such that the number of balls of radius $r$ needed to cover the ball  $B_{\T_{g}}(S,1)$ is at least $D^{g}$. From the basic covering theorem it follows that we can pack $D^{g}$ disjoint balls of radius $r/5$ in the big ball $B_{\T_{g}}(S,1)$. On the other hand, by the work of Muller and Puchta \cite{MP}, the number of different genus $g$ covers of $S_0$ is at most $Q^{g}$, for some universal constant $Q>0$. Thus if we choose $D>Q$, then for $\delta_0=r(D)$ there will be some ball from this disjoint collection that does not contain any cover of $S_0$. Since $\delta_0$ does not depend on the genus $g$ we are thus finished (observe that $\delta_0$ is a universal constant, and in particular it does not depend on $S_0$ or the injectivity radius of $S_0$).

\end{proof}

\subsection{Application of Theorem \ref{mainthmmg}: The lower bound for the diameter of $\mathcal{M}_{g,\epsilon}$} Let $\text{diam}(X)$ denote the diameter of a metric space $(X,d)$. 
In \cite{R-T}, the diameter (with respect to the Teichm\"uller metric) of $\text{diam}(\mathcal{M}_{g,\epsilon})$ was estimated. 

\begin{remark} There exists $\epsilon_M>0$ (called the Margulis constant) such that the the moduli space  $\mathcal{M}_{g,\epsilon}$ is connected for every $\epsilon<\epsilon_M$, and we make this assumption in the following discussion.  
\end{remark}

Rafi and Tao in \cite{R-T} have proved that there exists a universal constant $K>0$ such that 
\begin{equation}\label{diam}
\frac{1}{K} \log \frac{g}{\epsilon} \leq \text{diam}(\mathcal{M}_{g,\epsilon}) \leq K \log \frac{g}{\epsilon},
\end{equation}
for every $g$ (moreover, they prove the analogous result for punctured surfaces too).

We apply Theorem \ref{mainthmmg} to obtain  (for a fixed $0<\epsilon<\epsilon_M$) the lower bound in (\ref{diam}) as follows. Let $(X,d)$ denote a metric space. The following inequality holds in any metric space
$$
\eta_X(B_{X}(x,r+p),q) \leq \eta_X(B_{X}(x,r),q) \left( \sup\limits_{y \in X} \eta_X(B_X(y,p+q),q) \right),  
$$
for any $x \in X$, and any $p,q,r>0$. This inequality is proved as follows. Let $D_1,...,D_k$ be any covering of $B_X(x,r)$ by balls $D_i$ of radius $q$. Let $D'_i$ be the ball with the same center as $D_i$ but of radius $p+q$. Then the collection of balls $D'_1,...D'_k$ covers the bigger ball  $B_X(x,r+p)$. Now, we cover each ball $D'_i$ by balls of radius $q$ and taking infimums we finish the proof.

Letting $p=q=1$ in the above inequality, by induction we obtain the inequality 

\begin{equation}\label{diam-1}
\eta_X(X,1)=\eta_X(B_{X}(x,\text{diam}(X)),1) \leq \left( \sup\limits_{y \in X} \eta_X(B_X(y,2),1) \right)^{\text{diam}(X) +1}.
\end{equation}

Letting $X=\mathcal{M}_{g,\epsilon}$, and applying the lower bound in Theorem \ref{mainthmmg} and the upper bound in Theorem \ref{thmteich} we find that
$$
(c_1 g)^{2g} \leq   \eta(\mathcal{M}_{g,\epsilon},1) \leq {d_2}^{g ( \text{diam}(\mathcal{M}_{g,\epsilon})  +1 ) },
$$
and therefore the inequality 
$$
\frac{2(\log g + \log c_1)}{\log d_{2}} -1 \leq  \text{diam}(\mathcal{M}_{g,\epsilon}) 
$$
holds. This completes the proof.

\subsection{Punctured surfaces} By a punctured surface we mean a finite volume hyperbolic Riemann surface with finitely many cusps. The Ehrenpreis conjecture has been proved for closed Riemann surfaces. However, it is still open for punctured surfaces. 
Moreover, the method of proof that was used in the closed case  does not carry over to the punctured case. In fact, the punctured case of this conjecture might be an even more difficult problem. 

It seems reasonable to expect that the results of this paper  hold for punctured surfaces. However,  some of the methods we use in our proofs do not generalize to the punctured case in an obvious way. We mention two examples of this. 
First of all, throughout the paper we equip closed Riemann surfaces with geodesic triangulations that have an upper bound on the degree of a vertex (and on the number of vertices). Such triangulations do not exist for punctured surfaces. The second example  is Theorem \ref{thm23may}. The proof of this theorem relies heavily on the assumption that the Riemann surface in question is closed (an important point in this proof is that there is a lower bound on the injectivity radius). It is not entirely obvious if a satisfactory version of this  theorem holds for punctured surfaces.

\section{The upper bound}

In this section, we will prove Theorem \ref{upperbound}, namely the specific upper bound of $\eta( \mathcal{M}_{g,\epsilon},r_u)$ for some $r_u$ and all large $g$. This will be achieved by showing that every surface $S \in \mathcal{M}_{g,\epsilon}$ has a nice triangulation. If two surfaces have equivalent such triangulations, then the surfaces are not very far apart in moduli space. This reduces the problem to counting the number of equivalent triangulations, for which we can apply a result from \cite{KM}. We now make the notion of triangulations more precise.

\subsection{Triangulations of genus $g$ surfaces}
 
A \emph{genus $g$ triangulation} is a pair $(\tau, \iota)$, where $\tau$ is a connected graph and $\iota: \tau \to S_g$ is an embedding such that every component of $S_g \setminus \iota(\tau)$ is a topological disk that is bounded by three edges from $\iota(\tau)$. 

Two genus $g$ triangulations $(\tau_1,\iota_1)$ and $(\tau_2,\iota_2)$ are called \emph{equivalent} if there is a homeomorphism $h:S_g \to S_g$ such that $h(\iota_1(\tau_1)) = \iota_2(\tau_2)$, where $h$ maps vertices
and edges of $\iota_1(\tau_1)$ to vertices and edges of $\iota(\tau_2)$. We will write $\sim$ for this equivalence relation.
The set of genus $g$ triangulations is denoted by
$\Delta (g)$. Further, the subset $\Delta(k,g) \subset \Delta(g)$ are those triangulations for which the graph $\tau$ satisfies:
\begin{itemize}
\item each vertex of $\tau$ has degree at most $k$,
\item $\tau$ has at most $kg$ vertices and edges.
\end{itemize}

If the context is clear, we will sometimes confuse the graph $\tau$ and its image $\iota(\tau)$ on $S_g$.
The number of equivalence classes of triangulations in $\Delta(k,g)$ is bounded above, as the following result from \cite[Lemma 2.2]{KM} shows.

\begin{theorem}[\cite{KM}]
\label{kmlemma2}
There exists a constant $C>0$ depending only on $k$ such that for large $g$, we have
\[ |  \Delta(k,g)/\sim  | \leq (Cg)^{2g},\]
where $\Delta(k,g) / \sim$ denotes the set of equivalence classes of triangulations in $\Delta(k,g)$.
\end{theorem}

We say that a Riemann surface is $\epsilon$-thick if the shortest closed geodesic has length at least $\epsilon$. Every thick Riemann surface has a good triangulation in the sense of the following lemma.

\begin{lemma}[\cite{KM}]
\label{kmlemma1}
Let $S$ be an $\epsilon$-thick Riemann surface of genus $g \geq 2$. Then there exists $k=k(\epsilon)>0$, and a triangulation $(\tau, \iota) \in \Delta(k,g)$
that embeds in $S$ such that
every edge of $\iota(\tau)$ is a geodesic arc of length at most $\epsilon$ and at least $\epsilon /2$.
\end{lemma}

This lemma is proved in \cite[Lemma 2.1]{KM}, although we remark that the lower bound for the lengths of the edges of $\tau$ is proved but not explicitly stated.

\subsection{Proof of Theorem \ref{upperbound}}

We start with the following lemma on quasiconformal mappings between hyperbolic triangles proved in \cite[Lemma 3.1]{Bishop}.

\begin{lemma}[\cite{Bishop}]
\label{trianglelemma}
Let $T_1,T_2$ be hyperbolic triangles in $\D$ with angles $(\alpha_i,\beta_i,\gamma_i)$ for $i=1,2$ and opposite side lengths $(a_i,b_i,c_i)$ for $i=1,2$. Suppose that there exists $\theta >0$ such that all these angles are at least $\theta$. Suppose further that
\[ \max \left \{ \left | \log \frac{a_1}{a_2} \right |, \left | \log \frac{b_1}{b_2} \right |, \left | \log \frac{c_1}{c_2} \right | \right \} \leq A.\]
Then there exists a constant $K_0=K_0(\theta, A) \geq 1$ and a $K_0$-quasiconformal map $f:T_1 \to T_2$ which maps each vertex to the corresponding vertex and which is affine on each edge of $T_1$ with respect to the hyperbolic metric.
\end{lemma}

Using this lemma, we next show that equivalent triangulations in $\Delta(k,g)$ on different Riemann surfaces yield a quasiconformal mapping between the surfaces.

\begin{lemma}
\label{qcext}
Let $S_1,S_2$ be $\epsilon$-thick surfaces with triangulations $(\tau_1, \iota_1),(\tau_2, \iota_2) \in \Delta(k,g)$ and such that each edge of $\iota_1(\tau_1)$ and $\iota_2(\tau_2)$ is a geodesic arc and has length at least $\epsilon / 2$ and at most $\epsilon$.
If these two triangulations are equivalent, then there is a constant $K_0 \geq 1$ depending only on $\epsilon$ such that there exists a $K_0$-quasiconformal map $f:S_1 \to S_2$ such that $f(\iota_1(\tau_1)) = \iota_2(\tau_2)$.
\end{lemma}

\begin{proof}
Let $S_1,S_2 \in \mathcal{M}_{g,\epsilon}$ with equivalent triangulations $(\tau_1, \iota_1),(\tau_2, \iota_2) \in \Delta(k,g)$. This means there is a homeomorphism $h:S_1 \to S_2$ which maps vertices and edges of $\iota_1(\tau_1)$ to vertices and edges of $\iota_2(\tau_2)$. 
For any pair of hyperbolic triangles $T_1 \in \iota_1(\tau_1)$ and $T_2 \in \iota_2(\tau_2)$ with $h(T_1)=T_2$, 
replace $h$ by $f:T_1 \to T_2$, where $f$ is the quasiconformal mapping arising from Lemma \ref{trianglelemma}.

By definition, $T_1$ and $T_2$ both have side lengths between $\epsilon / 2$ and $\epsilon$. Since the angles of a hyperbolic triangle are completely determined by the lengths, this implies that there exists $\theta = \theta(\epsilon)>0$ such that all angles in $T_1$ and $T_2$ are at least $\theta$.
Therefore $f$ is $K_0$-quasiconformal for some $K_0$ depending only on $\epsilon$.
Define $f$ in this way for each pair of triangles related by $h$ in $\iota_1(\tau_1)$ and $\iota_2(\tau_2)$. Since $f$ is affine on each edge of $\iota_1(\tau_1)$, $f$ is well defined on all of $S_1$ and $K_0$-quasiconformal.
This proves the lemma.
\end{proof}

Given this lemma, the proof of the upper bound runs as follows.

\begin{proof}[Proof of Theorem \ref{upperbound}]
By Lemma \ref{kmlemma1}, every $S \in \mathcal{M}_{g,\epsilon}$ has a triangulation $(\iota,\tau) \in \Delta(k,g)$ for some $k=k(\epsilon)$, where each edge of $\iota(\tau)$ has length at least $\epsilon/2$ and at most $\epsilon$. 
For each equivalence class of triangulations that arises in this way, choose a representative $S_i \in \mathcal{M}_{g,\epsilon}$.
Then by Lemma \ref{qcext} and recalling the definition of the Teichm\"{u}ller metric, if $K>K_0$ the collection $B_{\mathcal{M}_g}( S_i, \log \sqrt{K})$ covers $\mathcal{M}_{g,\epsilon}$ where $K_0$, and hence $K$, depends only on $\epsilon$.

Now Theorem \ref{kmlemma2} implies there exists a constant $C=C(\epsilon)$ such that for all large $g$, the number of equivalence classes of triangulations in $\Delta(k,g)$ is at most $(Cg)^{2g}$. Therefore the covering number satisfies
\[ \eta \left ( \mathcal{M}_{g,\epsilon},  \log \sqrt{K} \right ) \leq (Cg)^{2g}.\]
This proves Theorem \ref{upperbound} with $r_u = \log \sqrt{K}$ and $c_u = C$.
\end{proof}

\section{The lower bound}

In this section, we will prove Theorem \ref{lowerbound}. The proof of this inequality is more involved than for the upper bound, and so we outline it here, before proving it in detail.

\begin{itemize}
\item Fix a base Riemann surface $S_0$ of genus $2$. Denote by $\epsilon$ the injectivity radius of $S_0$. Once and for all fix  a triangulation $\tau_0$ of $S_0$ with vertices $v_1,\ldots,v_n$, whose edges are geodesic arcs of length at most $\epsilon /10$ 
(see Lemma \ref{kmlemma1} above). We let $\chi=\chi(\tau_0)=\chi(S_0)>0$ denote a number such that every edge of $\tau_0$ is longer than $\chi$.
\item Consider the genus $g$ covers of $S_0$, and we remark that all covers in this section are unbranched. By a result of Muller and Puchta \cite{MP},
there exists a constant $P>0$ such that the number of such covers is at least $(Pg)^{2g}$ for large $g$.
\item Fix a cover $S_1 \in \mathcal{M}_{g,\epsilon}$, and let $\tau_1$ be a lift of the triangulation $\tau_0$ to $S_1$. For $i=1,\ldots,n$, label the vertices which are pre-images of $v_i$ by $\{w_{i,j}^1\}$ for $j=1,\ldots , (g-1)$. Denote the set of pre-images of $v_i$ in $S_1$ by $W^{S_1}_i$. Each vertex of $\tau_0$ has $g-1$ pre-images since the degree of the cover is $g-1$.
\item We show that if $S_2$ is another such cover, with the corresponding triangulation $\tau_2$, and $f:S_1 \to S_2$ is a quasiconformal map with small enough maximal dilatation that maps $W^{S_1}_i$ to $W^{S_2}_i$ for $i=1,\ldots,n$, then $S_1$ and $S_2$ are actually conformally equivalent.
\item There then exists a constant $K_2=K_2(S_0)$ such that if $S\in \mathcal{M}_{g,\epsilon}$ is any cover of $S_0$, then the ball
$B_{\mathcal{M}_g}(S,\log \sqrt{K_2})$ contains at most $D^g$ surfaces that are covers of $S_0$, where $D$ is a constant depending only on $S_0$.
\item Combining this with the Muller and Puchta estimate gives the lower bound.
\end{itemize}

\subsection{Covers of genus $2$ surfaces and lifts of triangulations}

In this subsection, we set some notation for the proof of the lower bound.

Let $S_0$ be an $\epsilon$-thick Riemann surface of genus $2$. We will consider the unbranched degree $d$ covers of $S_0$. By the Riemann-Hurwitz Theorem, if $f:S \to S_0$ is an unbranched covering of degree $d$, then the genus of $S$ satisfies
\[ g(S) = d+1.\]
In particular, the covers of $S_0$ which are of genus $g \geq 3$ correspond to degree $g-1$ covers.

We fix a triangulation $\tau_0$ of $S_0$ with vertices $v_1,\ldots,v_n$, and whose edges are geodesics arcs of length at most $\epsilon/10$ (for example, $\tau_0$ can be obtained by repeatedly applying barycentric subdivision of any given triangulation of $S_0$). 
By $\chi>0$ we denote a number such that each edge of $\tau_0$ is longer than $\chi$.

Let $S \in \mathcal{M}_{g,\epsilon}$ be a genus $g$ cover of $S_0$, so that $f:S \to S_0$ is a degree $g-1$ map. The triangulation $\tau_0$ lifts to a triangulation $\tau_S$ of $S$. Since the cover has degree $g-1$, the triangulation $\tau_S$ has $(g-1)n$ vertices. 

We put the vertices of $\tau_S$ into groups according to which vertex of $\tau_0$ they project to. More precisely, for $i=1,\ldots,n$, we write $W^S_i$ for the subset of vertices of $\tau_S$ given by
\[ W^S_i = f^{-1}(v_i).\]
Note that $W^S_i$ contains $g-1$ vertices of $\tau_S$. 

%\begin{figure}
%\includegraphics[height=80mm]{FKM_2.jpg}
%\caption{A genus $g$ cover $S$ of $S_0$ with the set $W_i^S$.\label{fig2}}
%\end{figure}

\begin{lemma}
\label{verticeslemma1}
There exists $K_1=K_1(S_0) \geq 1$ such that the following holds. If $S_1,S_2 \in \mathcal{M}_{g,\epsilon}$ are two genus $g$ covers of $S_0$, $f:S_1 \to S_2$ is a $K$-quasiconformal map with $K \leq K_1$ and $f$ maps $W^{S_1}_i$ to $W^{S_2}_i$ for $i=1,\ldots,n$, then $S_1$ and $S_2$ are conformally equivalent.
\end{lemma}

In proving this lemma, we will use the following elementary result, see for example \cite{FMbook}.

\begin{lemma}
\label{standardlemma}
Let $\delta >0$. There exists $K_{\delta}>1$ such that if $K<K_{\delta}$ and $f:X \to Y$ is a $K$-quasiconformal mapping between any two hyperbolic Riemann surfaces $X$ and $Y$, then for any $a,b \in X$ with $d_X(x,y) \geq \delta$, we have 
$$
\frac{1}{2} d_X(a,b) \leq d_Y(f(a),f(b)) \leq 2d_X(a,b) 
$$
where $d_X,d_Y$ denote the hyperbolic metrics on $X$ and $Y$ respectively.
\end{lemma}

\begin{proof}[Proof of Lemma \ref{verticeslemma1}]
Let $\tau_1,\tau_2$ be lifts of the triangulation $\tau_0$ of $S_0$. Suppose that $f:S_1 \to S_2$ is a $K$-quasiconformal map which satisfies $f(W_i^{S_1}) = W_i^{S_2}$ for $i=1,\ldots,n$.

Consider an edge $e$ in $\tau_1$ with vertices $w_1,w_2$, and suppose $w_1 \in W_i^{S_1}, w_2 \in W_j^{S_1}$ for some $i \neq j$. Then $f(w_1) \in W_i^{S_2}$ and $f(w_2) \in W_j^{S_2}$ by the hypothesis. Let $e'$ be the edge of $\tau_2$ that contains $f(w_1)$ and that is a lift of the same edge in $\tau_0$ as $e$. We claim that $f(e)$ is homotopic to $e'$ in $S_2$ modulo the endpoints, providing that $K\leq K_1 = K_{\chi}$, where $K_{\chi}$ is the constant from Lemma \ref{standardlemma} and $\chi$ is the lower bound on the length of edges from $\tau_0$.

First observe that any two points of $W^{S_2}_j$ must be at a distance at least $\epsilon$ apart since a geodesic arc joining them in $S_2$ projects to a closed curve in $S_0$ that is geodesic except possibly at one point. Each such arc in $S_0$ is homotopically non-trivial and thus has length at least $\epsilon$.

Next, since $w_1,w_2$ are a distance at most $\epsilon / 10$ apart, if $K_1$ is close enough to $1$ then by Lemma \ref{standardlemma}, $f(w_1),f(w_2)$ are at most $\epsilon / 5$ apart. 
On the other hand, if $f(w_2)$ is not an endpoint of $e'$, then the diameter of $f(e)$ is at least $9\epsilon/10$ since the distance from $f(w_2)$ to the endpoint of $e'$ in $W_j^{S_2}$ is at least $\epsilon$ (as observed in the previous paragraph) and $e'$ has length at most $\epsilon/10$. Therefore the endpoints of $e'$ are $f(w_1)$ and $f(w_2)$.

Finally, if $\gamma$ is an arc with the same endpoints as $e'$ and not homotopic to $e'$, then consider the concatenation of $\gamma$ and $e'$. By projecting this to $S_0$, we find that the concatenation is homotopically non-trivial and must have diameter at least $\epsilon/2$. Since the length of $e'$ is at most $\epsilon/10$, then $\gamma$ has diameter at least $2\epsilon/5$. From this it follows that $f(e)$ is homotopic to $e'$.

%\begin{figure}
%\includegraphics[height=80mm]{FKM_7.jpg}
%\caption{In this diagram, $\gamma_0$ is homotopic to $e'$ and so can be the image of $e$. The curve $\gamma_1$ connects $f(w_1)$ to a point of $W_j^{S_2}$ which is too far away and so cannot be the image of $e$. The curve $\gamma_2$ is not homotopic to $e'$ and so also cannot be the image of $e$.\label{fig7}}
%\end{figure}

We may therefore isotope $f$ to $\widetilde{f}$ such that $\widetilde{f}(e)=e'$, and repeat this procedure for each edge of $\tau_1$. We can further arrange that $\widetilde{f}$ is an isometry on each edge because every lift of an edge of $\tau_0$ has the same length. We may therefore replace $f$ by an isometry on the interior of each triangle, and produce an isometry $\widetilde{f}:S_1 \to S_2$. Hence we conclude $S_1$ and $S_2$ are conformally equivalent.
\end{proof}

\subsection{Number of covers in a ball in $\mathcal{M}_{g,\epsilon}$}

In this subsection we estimate the number of genus $g$ covers of $S_0$ that live in a particular ball in $\mathcal{M}_{g,\epsilon}$.

\begin{lemma}
\label{verticeslemma2}
Let $S_1\in\mathcal{M}_{g,\epsilon}$ be a genus $g$ cover of $S_0$. There exists $K_2=K_2(S_0)$ such that the ball $B_{\mathcal{M}_g}(S_1,\log \sqrt{K_2}) \subset \mathcal{M}_{g,\epsilon}$ contains at most $D^g$ surfaces that are covers of $S_0$, where $D=D(S_0)$ is a positive constant.
\end{lemma}

The strategy to prove this lemma is to first construct a fine enough grid $\Omega$, consisting of a finite number of points, in $S_1$.
Then we show that if $S_2$ is another cover of $S_0$ and $f:S_1 \to S_2$ is a $K$-quasiconformal map where $K$ is close enough to $1$, then we can replace $f$ with a homotopic quasiconformal map $\phi:S_1 \to S_2$ such that $\phi^{-1}$ maps vertices of $\tau_2$ to a subset of $\Omega$. Then we are able to associate a labeling of the grid $\Omega$ to $S_2$, with labels in a finite set.

It will follow from Lemma \ref{verticeslemma1} that if two covers $S_2$ and $S_3$ that are contained in the ball $B_{\mathcal{M}_g}(S_1,\log \sqrt{K_2})$ correspond to the same labeling of the grid $\Omega$, then $S_2$ and $S_3$ are isometric. This way we are able to estimate the number of covers of $S_0$ that live in $B_{\mathcal{M}_g}(S_1,\log \sqrt{K_2})$ by the number of labellings of the grid $\Omega$.

We now prove Lemma \ref{verticeslemma2} in detail.
 
\begin{proof}[Proof of Lemma \ref{verticeslemma2}]
For each $\delta>0$, let $\Omega_0(\delta)$ be a $\delta$-grid on $S_0$, that is, 
$\Omega_0(\delta)$ consists of a finite number of points and for every point $z \in S_0$,
\[ B_{S_0}(z,\delta ) \cap \Omega_0(\delta) \neq \emptyset.\]
Let $S_1 \in \mathcal{M}_{g,\epsilon}$ be a cover of $S_0$ and denote by $\Omega(\delta)$ the lift of $\Omega_0(\delta)$ to $S_1$.

Assume that $f:S_1 \to S_2$ is a $K$-quasiconformal map, where $S_2$ is another cover of $S_0$. Let $\tau_2$ be the lift of the triangulation $\tau_0$ to $S_2$ and let $\tau_2'$ be the triangulation $f^{-1}(\tau_2)$ in $S_1$ (observe that edges of $\tau_2'$ are not necessarily geodesic arcs).

Recalling the notation of Lemma \ref{standardlemma}, assume $K<K_{\chi}$. Then it follows from Lemma \ref{standardlemma} that for any vertex $z \in \tau_2'$, the ball $B_{S_1}(z, \chi/ 2)$
contains no other vertices of $\tau_2'$. Thus, the balls of radius $\chi/4$ centered at the vertices of $\tau'_2$ are mutually disjoint. Let $\delta < \chi / 4$ and let $\omega \in \Omega (\delta)$ be a point of the grid $\Omega(\delta)$ such that $d_{S_1}(\omega,z)<\delta$. We can find a map that is supported on the ball $B_{S_1}(z, \chi/4)$ that maps $z$ to $\omega$, described in the following lemma. 

\begin{lemma}[\cite{Reich}]
\label{reichlemma}
Let $a>\delta>0$ and $p \in \D$. There exists a function $K(a,\delta)>1$ such that for every point $q \in B_{\D}(p,a)$ with $d(p,q) \leq \delta$, there exists a $K(a,\delta)$-quasiconformal map $f:B_{\D}(p,a) \to B_{\D}(p,a)$ 
that is the identity on $\partial B_{\D}(p,a)$ and $f(p)=q$. Moreover, for all $a$, $\lim_{\delta \to 0} K(a,\delta) = 1$.
\end{lemma}

We remark that Edgar Reich in \cite{Reich} computed the function $K(a,\delta)$.

Let $h :S_1 \to S_1$ be the $K(\chi/4, \delta)$-quasiconformal map arising from Lemma \ref{reichlemma} that is supported on the disjoint union of the balls of radius $\chi/4$ centered at the vertices of $\tau_2'$ and that maps vertices of $\tau_2'$ to the grid $\Omega(\delta)$. Let $\phi:S_1 \to S_2$ be given by $f \circ h^{-1}$. 

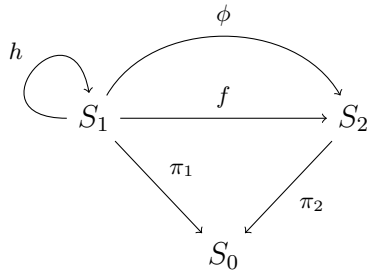
\begin{center}
\begin{tikzpicture}[description/.style={fill=white,inner sep=2pt}]
\matrix (m) [matrix of math nodes, row sep=3em,
column sep=2.5em, text height=1.5ex, text depth=0.25ex]
{S_1 & & S_2 \\
& S_0 & \\ };
\path[->,font=\scriptsize]
(m-1-1) edge node[auto] {$ f $} (m-1-3)
(m-1-1) edge node[auto] {$ \pi_1 $} (m-2-2)
(m-1-3) edge node[auto] {$ \pi_2 $} (m-2-2);
\path[->,font=\scriptsize] (m-1-1) edge [out=180, in=100, loop] node[auto]{$ h $} (m-1-1);
\path[->,font=\scriptsize] (m-1-1) edge [out=60, in=120] node[auto]{$ \phi $} (m-1-3);
\end{tikzpicture}
\end{center}

Then $\phi$ is $K \cdot K(\chi/4, \delta)$-quasiconformal mapping and $\phi^{-1}$ maps vertices of $\tau_2$ to $\Omega(\delta)$.
By Lemma \ref{reichlemma}, we can choose $\delta$ small enough such that $K(\chi/4, \delta) \leq K_1^{1/4}$, where $K_1=K_1(S_0)$ is the constant from Lemma \ref{verticeslemma1}. Set $K_2 = K_1^{1/4}$. Then if we choose $K \le K_2$, we may conclude that
$\phi$ is $\sqrt{K_1}$-quasiconformal.

We label the grid $\Omega(\delta)$ in $S_1$ as follows. Label $\omega \in \Omega(\delta)$ by $i \in \{1,\ldots,n\}$ if $\phi(\omega) \in W^{S_2}_i$, recalling that $W^{S_2}_i$ is the set of pre-images in $S_2$ of the vertex $v_i \in \tau_0$. 
Otherwise, we label $\omega \in \Omega(\delta)$ by $0$. In this way, we associate a labeling of $\Omega(\delta)$ to every element of $B_{\mathcal{M}_{g}}(S_1,\log \sqrt{K_2})$.

If $S_3$ is another cover of $S_0$ in $B_{\mathcal{M}_g}(S_1, \log \sqrt{K_2})$, then let $\widetilde{\phi}$ denote the corresponding map $\widetilde{\phi}:S_1 \to S_3$. 
If the maps $\phi, \widetilde{\phi}$ induce the same labeling of the grid $\Omega(\delta)$, then the map $\widetilde{\phi} \circ \phi^{-1}:S_2 \to S_3$ is $K_1$-quasiconformal and therefore by Lemma \ref{verticeslemma1}  this map is homotopic to an isometry. But this means that $S_2$ and $S_3$ determine the same point of $\mathcal{M}_{g}$. Thus, the number of covers of $S_0$ in $B_{\mathcal{M}_{g}}(S_1,\log \sqrt{K_2})$ is bounded above by the number of labellings of the grid.

The grid $\Omega_0(\delta)$ contains finitely points, say $N$. Since $S_1$ is a genus $g$ cover of $S_0$, there are $(g-1)N$ points of the grid $\Omega(\delta)$ in $S_1$.
The number of possible labellings of $\Omega(\delta)$ is the number of labels raised to the number of points in the grid, that is, at most 
\[(n+1)^{N(g-1)}\leq D^g,\]
for some constant $D$ depending only on $\epsilon$, $S_0$ and the triangulation $\tau_0$.
This completes the proof.
\end{proof}

\subsection{Proof of Theorem \ref{lowerbound}}

By the results of Muller and Puchta \cite{MP}, there are at least $(Pg)^{2g}$ genus $g$ covers of $S_0$ for large $g$. By Lemma \ref{verticeslemma2}, every ball $B_{\mathcal{M}_{g}}(S,\log \sqrt{K_2})$ contains at most $D^g$ other genus $g$ covers of $S_0$, where $K_2=K_2(S_0)$ and $D=D(S_0)$.

Hence the number of balls of radius $\log \sqrt{K_2}$ needed to cover $\mathcal{M}_{g,\epsilon}$ is at least
\[ \eta(\mathcal{M}_{g,\epsilon} , \log \sqrt{K_2} ) \geq \frac{(Pg)^{2g}}{D^g} = (c_lg)^{2g},\]
for large $g$, where $c_l = P/D^{1/2}$ depends only on $S_0$. This proves Theorem \ref{lowerbound} with $r_l = \log \sqrt{K_2}$.

\section{Covering numbers, Teichm\"uller spaces and the proof of Theorem \ref{thmteich}}

The goal of this section is to show how to extend the particular inequalities for the covering number of $\mathcal{M}_{g,\epsilon}$ given by Theorems \ref{upperbound} and \ref{lowerbound} to all values of $r$. To this end, we will estimate the covering number for balls in Teichm\"{u}ller space $\T_g$. This metric space is not immediately amenable to estimating the cover number and so to simplify the task, we will use the Bers embedding of $\T_g$ into a Banach space where we are able to estimate the covering number. We also need to show that the covering number is well-behaved with respect to the Bers embedding.

\subsection{Basic properties of covering numbers}

A bi-Lipschitz homeomorphism $f:X \rightarrow Y$ between metric spaces is a mapping $f$ such that $f$ and $f^{-1}$ satisfy a uniform Lipschitz condition, that is, there exists $L \geq 1$ such that
\begin{equation*}
\frac{ d_{X}(x,y) }{L} \leq d_{Y}(f(x),f(y)) \leq Ld_{X}(x,y)
\end{equation*}
for all $x,y \in X$. The smallest such constant $L$ is called the {\it isometric distortion} of $f$.

\begin{theorem}[Quasi-invariance of covering numbers]
\label{thm2}
Let $f:(X,d_X) \to (Y,d_Y)$ be a  $L$-bi-Lipschitz homeomorphism between metric spaces.
Then for all $E \subset X$ and $r>0$, 
\[ \eta_X (X,r) \leq \eta_Y \left (Y, \frac{r}{L} \right ).\]
\end{theorem}

\begin{proof}
Let $r_1 = r/L$, and cover $Y$ by finitely many balls $\widetilde{B_i} = B(y_i,r_1)$, for $i=1,\ldots,n$. Then since $f$ is surjective and $L$-bi-Lipschitz, we have
\[ f^{-1}(\widetilde{B_i}) \subset  B(f^{-1}(y_i),Lr_1 ) .\]

To see this, suppose that $y_i = f(x_i)$ and $d_Y(y,y_i) < r_1$. Then $y=f(x)$ for some $x \in X$, and we have
\[ d_X(x,x_i) \leq L d_Y(f(x),f(x_i)) < Lr_1.\]
For $i=1,\ldots,n$, let $B_i = B(x_i,r)$. Then 
\[ E \subset \bigcup _{i=1}^n f^{-1}(\widetilde{B_i}) \subset \bigcup _{i=1}^n B_i,\]
and we are done.
\end{proof}

The following result gives the lower bound for the number of small balls needed to cover a bigger ball in an $n$-dimensional Banach space.

\begin{lemma}\label{banach} Let $(X,|| \cdot ||)$ denote a $n$-dimensional real Banach space. Then for $R>r>0$ we have
$$
\eta_X(B_{X}(0,R),r) \ge \left( \frac{R}{r} \right)^{n}.
$$
\end{lemma}

\begin{proof} We identify $X$ with $\R^n$ endowed with the norm $||\cdot ||$. Let $\text{Vol}$ denote the standard volume on $\R^n$. Let $A:\R^n \to \R^n$ denote the dilation $A(x_1,...,x_n)=(R/r)(x_1,...,x_n)$. Then 
$A(B_{X}(0,r))=B_{X}(0,R)$. Thus
$$
\text{Vol}(B_{X}(0,R))=\left( \frac{R}{r} \right)^{n} \text{Vol}(B_{X}(0,r)),
$$
so we need at least $(R/r)^{n}$ balls of radius $r$ in $(\R^{n},||\cdot ||)$ to cover a ball of radius $R$ in $(\R^{n}, ||\cdot ||)$, and we are finished.
\end{proof}

\subsection{Teichm\"uller space is locally bi-Lipschitz equivalent to the Bers space} Let $M$ denote an arbitrary  hyperbolic Riemann surface. By $\T(M)$ we denote the Teichm\"uller space of marked Riemann surfaces that are quasiconformally equivalent to $M$. The space $\T(M)$ is endowed with the Teichm\"uller metric $d_{\T(M)}$. By $Q(M)$ we denote the Banach space of holomorphic quadratic differentials on $M$ endowed with the supremum norm
\[ 
\av \av \varphi \av \av _{Q(M)} = \sup _{z \in M} \rho_M^{-2}(z) \av \varphi (z) \av,
\]
where $\rho_M$ is the density of the hyperbolic metric on $M$ (although the expression  $\rho_M^{-2}(z) \av \varphi (z) \av$ is given in local charts on $M$ it represents a well defined function on $M$). The corresponding distance in the Banach space $Q(M)$ is denoted by $d_{Q(M)}$. We call $Q(M)$ the Bers space associated to the surface $M$. 

Let $\beta_{M} :\T(M) \to Q(M)$ denote the Bers embedding of $\T(M)$ into $Q(M)$ with respect to the base point $M$. The next (most probably well known) theorem states that the restriction of  $\beta_{M}$ on a ball of a fixed radius in $\T(M)$ is a bi-Lipschitz mapping onto its image. 

\begin{theorem}\label{prop10} Let $R>0$ and  $B_{\T(M)}(M,R) \subset \T(M)$ denote the ball of radius $R$ and centered at $M$. There are positive constants $b_l$ and $b_u$, where $b_l$ is a universal constant and $b_u$ depends only on $R$, 
such that 
\begin{equation}\label{jednacina}
b_l d_{Q(M)}(\beta_M(\tau_1),\beta_M(\tau_2))\le d_{\T(M)}(\tau_1,\tau_2) \le  b_u d_{Q(M)}(\beta_M(\tau_1),\beta_M(\tau_2)),
\end{equation}
for any $\tau_1,\tau_2 \in B_{\T(M)}(M,R)$. Moreover, there is a universal function  $a(R)>0$, $R>0$ (that is $a(R)$ does not depend on $M$),  such that 
\begin{equation}\label{jednacina-druga}
B_{Q(M)}(0,a(R)) \subset \beta_M(B_{\T(M)}(M,R)) \subset B_{Q(M)}(0,6).
\end{equation}

\end{theorem}

\begin{proof} When $M=\D$, where $\D$ is the unit disc in the complex plane,  these inequalities are  known. We have
$$
\widehat{b}_l d_{Q(\D)}(\beta_{\D}(\tau_1),\beta_{\D}(\tau_2))\le d_{\T(\D)}(\tau_1,\tau_2) \le  \widehat{b}_u d_{Q(\D)}(\beta_{\D}(\tau_1),\beta_{\D}(\tau_2)),
$$
for any $\tau_1,\tau_2 \in B_{\T(\D)}(\D,R)$,  where $\widehat{b}_l$ is a universal constant and $\widehat{b}_u$ depends only on $R$.   
In this case, the lower bound is the formula $(4.4)$ from  page 113 in \cite[(III.4.2)]{L} and the upper bound is the formula  $(4.6)$ from  page 113 in \cite[(III.4.2)]{L} (as pointed out on page 113 in \cite[(III.4.2)]{L}, one can take $\widehat{b}_l=12$).

On the other hand, the Bers embedding $\beta_{\D}$ is an open mapping (see \cite{L}, \cite{FMbook}) and thus there exists a function $\widehat{a}(R)>0$
$$
B_{Q(\D)}(0, \widehat{a}(R)) \subset \beta_{\D}(B_{\T(\D)}(\D,R)).
$$
The inclusion
$$
\beta_{\D}(B_{\T(\D)}(\D,R)) \subset B_{Q(\D)}(0,6), 
$$
is Nehari's estimate and the inclusion.

We extend these inequalities to an arbitrary hyperbolic Riemann surface as follows. Recall the embeddings $\iota_{\T}:\T(M) \to \T(\D)$ and  $\iota_{Q}:Q(M) \to Q(\D)$ such that 
$\iota_Q \circ \beta_M = \beta_{\D} \circ \iota_{\T}$ (see for example \cite{FMbook}, \cite{L}, \cite{MS}). Then (see \cite[Theorem 3.3]{MS})
$$
d_{\T(\D)}(\iota_{\T}(\tau_1), \iota_{\T}(\tau_2)) \leq d_{\T(M)}(\tau_1, \tau_2) \leq 3 d_{\T(\D)}(\iota_{\T}(\tau_1), \iota_{\T}(\tau_2)),
$$
for all $\tau_1,\tau_2 \in \T(M)$. On the other hand, we have $d_{Q(M)}(\phi,\psi)=d_{Q(\D)}(\iota_Q(\phi),\iota_Q(\psi))$. Set $b_l=\widehat{b}_l$ and $b_u=3\widehat{b}_u$ and the inequalities from (\ref{jednacina}) follow. 
Similarly, set 
$$
a(R)=\widehat{a}(R/3).
$$
Then the inclusions (\ref{jednacina-druga})  follow and we are finished.
\end{proof}

\subsection{Estimates for covering numbers of balls in $\T_g$} Applying Theorem \ref{prop10} of the previous subsection to a genus $g \ge 2$ Riemann surface $S$ we find
\begin{equation}\label{jednacina-1}
b_l d_{Q(S)}(\beta_S(\tau_1),\beta_S(\tau_2))\le d_{\T_{g}}(\tau_1,\tau_2) \le  b_u d_{Q(S)}(\beta_S(\tau_1),\beta_S(\tau_2)),
\end{equation}
for any $\tau_1,\tau_2 \in B_{\T_{g}}(S,R)$, and
\begin{equation}\label{jednacina-druga-1}
B_{Q(S)}(0,a(R)) \subset \beta_S(B_{\T_{g}}(S,R)) \subset B_{Q(S)}(0,6), 
\end{equation}

Combining these results with Theorem \ref{thm2} we derive the following theorem:

\begin{theorem}
\label{cor2}
There exists a constant $L = L(R)$ such that

\begin{equation}\label{glava}
\eta_{Q(S)}(B_{Q(S)}(0,a(R)), 2Lr)  \leq \eta_{\T_{g}} (B_{\T_{g}}(S,R),r) \leq \eta_{Q(S)} ( B_{Q(S)}(0,6),\frac{r}{2L}),
\end{equation}
where $a(R)$ is the  function from Theorem \ref{prop10}
\end{theorem}

\begin{proof} We have from Theorem \ref{prop10} that the restriction of  $\beta_S$ is $L=L(R)$ bi-Lipschitz on the ball $B_{\T_{g}}(S,R)$. Then  the inequality
$$
\eta_{B_{\T_{g}}} (B_{\T_{g}}(S,R),r) \leq \eta_{E } (E ,\frac{r}{L})
$$
follows from Theorem \ref{thm2}, where $E=\beta_S(B_{\T_{g}}(S,R))$. Now the second inequality in (\ref{glava}) follows from the previous inequality and (\ref{obvfact}) and  (\ref{obvfact-1}). The first inequality in (\ref{glava})  is proved similarly.
\end{proof}

\subsection{The proof of Theorem \ref{thmteich}} Recall that Theorem \ref{thmteich} states that if $S \in \T_g$, $R>0$ and $B_{\T_{g}}(S,R)$ is a ball in $\T_g$, where $S$ has injectivity radius $\epsilon >0$, then there exist $d_1=d_1(R,r) \ge 0$ and 
$d_2=d_2(\epsilon,R,r)>0$ such that
\[ d^{g}_1 \leq  \eta_{\T_{g}}(B_{\T_{g}}(S,R),r) \leq d^{g}_{2},\]
for large $g$ and where $d_1(R,r) \to \infty$ when $r \to 0$ and  when $R$ is fixed.
 
The lower bound 

\[ d^{g}_1 \leq  \eta_{\T_{g}}(B_{\T_{g}}(S,R),r) ,\]
follows from the first inequality in (\ref{glava}) and Lemma \ref{banach}. 

To prove the second inequality
\[ \eta_{\T_{g}}(B_{\T_{g}}(S,R),r) \leq d^{g}_2,\]
we observe that from (\ref{glava}) we have 
\begin{equation}\label{ova}
\eta_{\T_{g}} (B_{\T_{g}}(S,R),r) \leq \eta_{Q(S)} ( B_{Q(S)}(0,6),\frac{r}{2L}),
\end{equation}
for $L=L(R)$. Thus, it is enough to be able to estimate the corresponding covering number in the linear space $Q(S)$.
The following theorem is proved in the last section.

\begin{theorem}[Covering number for $Q(S)$]
\label{thm3}
Let $S$ be a closed surface of genus $g$ and injectivity radius $\epsilon >0$. Let
$B_{Q(S)}(0, R)$ be a ball of radius $R$ in $Q(S)$ centered at $0 \in Q(S)$. Then there exists $c=c(\epsilon,R,r)$ such that
\[ \eta_{Q(S)} (B_{Q(S)} (0, R),r) \leq  c^{g} .\]
\end{theorem}

From (\ref{ova}) and Theorem \ref{thm3} we have
$$
\eta_{\T_{g}} (B_{\T_{g}}(S,R),r) \leq \eta_{Q(S)} ( B_{Q(S)}(0,6),\frac{r}{2L}) \leq d^{g}_2,
$$
for some $d_2=d_2(\epsilon,R,r)$. This completes the proof of Theorem \ref{thmteich}.

\section{Covering number for $Q(S)$}

Let $S$ be a closed surface of genus $g\geq 2$ and injectivity radius $\epsilon$ that is fixed throughout this section. The aim of this section is to prove Theorem \ref{thm3} by computing $\eta_{Q(S)}(E,r)$ when $E$ is a ball.
We consider $\C^n$ as the Banach space equipped with the supremum norm. 

\begin{theorem}
\label{thm23may} Let $S$ be a genus $g$ closed Riemann surface of injectivity radius $\epsilon$. There exist a universal constant $\alpha \ge 1$, a constant $K=K(\epsilon)$, and a $\alpha$-bi-Lipschitz linear isomorphism $F:Q(S) \to V$, where $V$ is a linear subspace of $\C^{n}$, for some  $n \le Kg$.
\end{theorem}

Assuming this theorem for the moment, it reduces Theorem \ref{thm3} to estimating the covering number for balls in $\C^n$ with the supremum norm, which we do in the following lemma.

\begin{lemma}
\label{lemma5}
Let $B_{\C^{n}}(0,R)$ be the ball of radius $R$ in $\C^n$. Then
\[ \eta (B_{\C^{n}}(0,R),r) \leq \left ( \frac{2\sqrt{2}R}{r} +2 \right )^{2n}.\]
\end{lemma}

\begin{proof} Consider the Banach space $\R^{2n}$ with its supremum norm. The vector spaces $\C^{n}$ and $\R^{2n}$ are isomorphic but are not isometric considering the respective supremum norms. However they are $\sqrt{2}$-bi-Lipschitz equivalent.   
Balls in $\R^{2n}$ with the supremum norm are geometrically cubes, which is not the case for $\C^n$ with the supremum norm. Since it is a little  easier to to get explicit estimates with geometric cubes, we first prove the statement of the lemma for the ball $B_{\R^{2n}}(0,R) \subset \R^{2n}$ with the supremum norm and then use the fact that $\R^{2n}$ is $\sqrt{2}$-bi-Lipschitz equivalent to $\C^n$ and Theorem \ref{thm2}.

We have $B_{\R^{2n}}(0,R) = [-R,R]^{2n} \subset \R^{2n}$. It is elementary to check that for any $\delta >0$, the union of balls
\[ \bigcup _{(i_{1},\ldots,i_{m}) \in \Z^{m}, \av i_{j} \av \leq C }  B_{\R^{2n}}((2 i_{1}r,\ldots ,2 i_{m}r),r + \delta) \]
covers the ball $B_{\R^{2n}}(0,R)$ if 
\[ C = \left \lceil \frac{1}{2} \left ( \frac{R}{r+\delta} - 1 \right ) \right \rceil, \]
where $\lceil y \rceil$ denotes the smallest integer greater or equal to $y \in \R$.

It then follows that the number of balls in this covering is 
\begin{align*}
(2C +1) ^{2n} &= \left ( 2\left [ \frac{1}{2} \left ( \frac{R}{r+\delta} - 1 \right ) \right ] +1 \right) ^{2n} \\ &\leq \left ( \frac{R}{r+\delta} +2 \right )^{2n} \\ &< \left ( \frac{R}{r} +2 \right )^{2n}.
\end{align*}
Using the fact that $\C^n$ with its supremum norm is $\sqrt{2}$-bi-Lipschitz equivalent to $\R^{m}$  with its supremum norm, and combining Theorem \ref{thm2} with (\ref{obvfact}) and  (\ref{obvfact-1}) completes the proof. 
\end{proof}

We can now prove Theorem \ref{thm3}.

\begin{proof}[Proof of Theorem \ref{thm3}]
By Theorem \ref{thm23may}, $Q(S)$ is $\alpha$-bi-Lipschitz equivalent
to a subspace $V$ of $\C^n$, where $n \leq Kg$ and $K$ depends only on the injectivity radius $\epsilon$ of $S$. Recall the mapping $F:Q(S) \to V$ from Theorem \ref{thm23may}.
By Theorem \ref{thm2}, (\ref{obvfact}) and  (\ref{obvfact-1}) 
\[ \eta _{Q(S)} (B_{Q(S)}(0,R),r) \leq \eta _V \left( F(B_{Q(S)}(0,R)),\frac{r}{2\alpha} \right ).\]
Since $F(B_{Q(S)}(0,R))$ is contained in the ball $B_V(0,R\alpha)$ in $V$, \eqref{obvfact} implies that
\[ \eta _V \left( F(B_{Q(S)}(0,R)),\frac{r}{2\alpha} \right ) \leq
\eta_V \left( B_V(0,R\alpha),\frac{r}{2\alpha} \right ).\]
Since $V$ is a subspace of $\C^n$, 
it follows from Proposition \ref{obvfact-1} that
\[ \eta_V \left( B_V(0,R\alpha),\frac{r}{\alpha} \right ) \leq \eta _{\C^{n}} \left( B_{\C^{n}}(0,R\alpha) ,\frac{r}{4\alpha} \right ).\]
Finally, Lemma \ref{lemma5} implies that
\[ \eta _{Q(S)} (B_{Q(S)}(0,R),r) \leq \left ( \frac{4\sqrt{2} \alpha^2 R}{r} +2 \right )^{2n}.\]
\end{proof}

It remains to prove Theorem \ref{thm23may}.
Recall that $S$ is a genus $g$ Riemann surface of injectivity radius at least $\epsilon$. The following is a brief outline of the construction of the linear map $F:Q(S) \to \C^{n}$ from Theorem \ref{thm23may}.

\begin{enumerate}
\item Estimate how many small balls are needed to cover $S$ in terms of the genus and injectivity radius.
\item Show the expression $\rho_S^{-2} \varphi$ does not vary too much on these balls for any $\varphi \in Q(S)$ of norm at most $1$.
\item Decompose $S$ into small subregions and define $F$ by averaging $\rho_S^{-2} \varphi$ on each small subregion.
\end{enumerate}

\begin{remark}We remark that a similar operator to $F$ was constructed in \cite{F2} for infinite type surfaces.
\end{remark}

We first estimate how many small balls it takes to cover $S$.

\begin{lemma}\label{lemma1} Let $\delta >0$. Then there exists $K=K(\epsilon,\delta)$ such that any genus $g$ Riemann surface $S$ of injectivity radius at least $\epsilon$  can be covered by $n$ balls of radius $\delta$, for some  $n\leq Kg$.
\end{lemma}

\begin{proof} Let $\tau \in \Delta(k,g)$ be a triangulation of $S$ given by Lemma \ref{kmlemma1}, and let
$\Delta$ be a geodesic triangle which is a face in $\tau$. Then $\Delta$ is contained in a ball of radius $\epsilon$ by Lemma \ref{kmlemma1}. Let $C = C(\epsilon,\delta)$ denote a constant such  that every ball of hyperbolic radius
$\epsilon$ in the hyperbolic disc $\D$ can be covered by at most $C$ balls of radius $\delta$. 

Let $V,E,F$ denote the number of vertices, edges and faces of $\tau$ respectively.
The Euler characteristic of $S$ is $\chi(S) = 2-2g$, and so we have by Lemma \ref{kmlemma1} that
\[ 2-2g = V-E+F \geq F-kg\]
and so 
\[ F \leq 2+(k-2)g\]
provides an upper bound for the number of faces in $\T$.
Using this, the number of balls of radius $\delta$ needed to cover $S$ is at most
\[ (2+(k-2)g)C \leq 2kCg,\]
and where $K=2kC$ depends only on $\epsilon$ and $\delta$.
\end{proof}

Now we know how many small balls are needed to cover $S$, we need to know how much the expression in the definition of the Bers norm varies over these balls.

\begin{lemma}
\label{lemma2}
Given $\xi>0$, there exists $\delta>0$ depending on $\xi$ such that
\[ \left \av \rho ^{-2}(z)\varphi (z) - \rho ^{-2}(w)\varphi (w) \right \av < \xi, \, \, \text{for every} \, \, z,w \in \D \]
for any $\varphi \in Q(\D)$ with $\av \av \varphi \av \av _{Q(\D)} \leq 1$, whenever $d(z,w)<\delta$, where $d$ and $\rho$ denote the hyperbolic metric and
hyperbolic density on $\D$ respectively.
\end{lemma}

\begin{proof}
Let $\xi >0$ and $\varphi \in Q(\D)$ with $\av \av \varphi \av \av _Q \leq 1$. 
Suppose that $\delta < 1/2$ and let  $B_{\D}(0,\delta) = \{ t \in \C : \av t \av <\delta\}$. 
Without loss of generality we may  assume that $z,w \in B_{\D}(0,\delta)$. 

For $\av p \av \leq 1/2$, we have
\begin{equation}
\label{l2eq1}
\av \varphi (p) \av \leq \rho^{-2}(p) \leq \frac{64}{9}.
\end{equation}
Thus, by the Cauchy integral formula and \eqref{l2eq1},
\begin{align*}
\av \varphi(z) - \varphi(w) \av &= \frac{1}{2\pi} \left \av \int_{\av t \av = 1/2} \left ( \frac{\varphi(t)}{z-t} - \frac{\varphi(t)}{w-t} \right ) \: dt \right \av \\
&\leq \frac{128\av z- w \av}{9(1-2 \delta)^2}
\end{align*}
and hence there exists a constant $C_1$ such that
if $z,w, \in B_{\D}(0,\delta)$, we have
\begin{equation}
\label{l2eq2}
\av \varphi(z) - \varphi (w) \av \leq C_1 \delta.
\end{equation}
We also have by elementary calculations that for $z,w, \in B_{\D}(0,\delta)$, there exists a constant
$C_2$ such that
\begin{equation}
\label{l2eq3}
\left \av \rho^{-2}(z) - \rho^{-2}(w) \right \av < C_2 \delta.
\end{equation}
Therefore, if $z,w, \in B_{\delta}$, by \eqref{l2eq2} and \eqref{l2eq3} we have
\begin{align*}
\left \av \rho ^{-2}(z)\varphi (z) - \rho ^{-2}(w)\varphi (w) \right \av &\leq \rho^{-2}(z) \av \varphi(z) - \varphi (w) \av + \av \varphi(w) \av \cdot \av \rho^{-2} (z) - \rho^{-2} (w) \av \\
&\leq \frac{(1-\delta^2)^2C_1 \delta}{4} + \frac{64C_2 \delta}{9}.
\end{align*}
Hence if $\delta$ is chosen small enough, we have 
\[ \left \av \rho ^{-2}(z)\varphi (z) - \rho ^{-2}(w)\varphi (w) \right \av < \xi .\]
\end{proof}

\begin{proof}[Proof of Theorem \ref{thm23may}]
We now are in a position to construct $F$.
Choose any $\xi >0$. Then choose $\delta >0$ such that the conclusions of Lemma \ref{lemma2}
are satisfied. Then use Lemma \ref{lemma1} to find a covering $B_1,\ldots, B_n$ of $S$
by balls of radius $\delta$, where $n \leq Kg$ and $K = K(\epsilon,\xi)$, since
$\delta$ depends on $\xi$. Moreover, for each ball $B_i$ we choose a coordinate chart $\widetilde{B_i}$.

For $j=1,\ldots, n$ choose $p_j \in B_j$ and denote by   $z_j \in \widetilde{B}_j$ the corresponding points (that is, $z_j \in \widetilde{B}_j$ represents the point $p_j \in S$).
Then for $\phi \in Q(S)$ we let $F(\phi)=(f_1,...,f_n) \in \C^{n}$, where for each $j=1,...,n$, we let
$$
f_j=(\rho^{-2}_S \phi)(z_j),
$$
and $\rho^{-2}_S\phi$ is the corresponding function in the chart $\widetilde{B}_j$.

\begin{remark} The map $F:Q(S) \to \C^{n}$ depends on the choice of charts $\widetilde{B}_j$ because $\rho^{-2}_S \phi$ is not a function on $S$, but it is a $(-1,1)$ complex form on $S$. As we observed before, the expression $|\rho^{-2}_S \phi|$ is a function on $S$, thus the norm of the linear map $F$ does not depend on the choice of charts  $\widetilde{B}_j$. 
\end{remark} 

By definition we have 
$$
||F(\phi)||_{\infty} \leq ||\phi||_{Q(S)}.
$$
On the other hand, it follows from Lemma \ref{lemma2} that for any $\phi \in Q(S)$, $||\phi||_{Q(S)}=1$, we have
$$
(1-\xi)||\phi||_{Q(S)} \leq ||F(\phi)||_{\infty}. 
$$
This shows that $F$ is an $\alpha$-bi-Lipschitz linear map onto its image with $\alpha = (1-\xi)^{-1}$. Since $\xi>0$ was arbitrary we are finished.
\end{proof}

\end{document}